\documentclass{article}
\usepackage[utf8]{inputenc}
\usepackage[margin=1in]{geometry}
\usepackage{amsmath} 
\usepackage{amsthm}
\usepackage{todonotes}
\usepackage[nolist]{acronym}
\usepackage{xcolor}
\usepackage{amsmath}
\usepackage{amssymb}
\usepackage{mathtools}
\usepackage{subcaption}
\usepackage{hyperref}
\hypersetup{colorlinks=true, linkcolor=black}
\usepackage[ruled]{algorithm2e}
\usepackage{mathrsfs}
\definecolor{ccolor}{RGB}{203,96,21}

\newcommand{\R}{\mathbb{R}}

\newcommand{\N}{\mathbb{N}}

\newcommand{\0}{\mathbf{0}}
\newcommand{\1}{\mathbf{1}}

\newcommand{\dc}{\mathcal{D}}
\newcommand{\pc}{\mathcal{P}}

\DeclarePairedDelimiter{\abs}{\lvert}{\rvert}
\DeclarePairedDelimiter{\norm}{\lVert}{\rVert}

\DeclarePairedDelimiter{\diag}{\textrm{diag}(}{)}

\DeclarePairedDelimiterX{\inp}[2]{\langle}{\rangle}{#1, #2}

\DeclareMathOperator*{\find}{find}

\newtheorem{thm}{Theorem}[section]
\newtheorem{lem}[thm]{Lemma}
\newtheorem{prop}[thm]{Proposition}
\newtheorem{prob}[thm]{Problem}
\newtheorem{cor}{Corollary}

\newtheorem{rmk}{Remark} 
\newcommand{\bx}{\mathbf{X}}
\newcommand{\bu}{\mathbf{U}}

\usepackage{algorithmicx}

\newlength{\exfiglength}
\renewcommand\footnotemark{}

\title{\LARGE \bf Data-Driven Superstabilization of Linear Systems under Quantization}

\author{Jared Miller$^{1, 2}$,  Jian Zheng$^2$, Mario Sznaier$^2$, Chris Hixenbaugh$^3$
\thanks{$^1$J. Miller is with the Automatic Control Laboratory (IfA), Department of Information Technology and Electrical Engineering (D-ITET), ETH Z\"{u}rich, Physikstrasse 3, 8092, Z\"{u}rich, Switzerland (e-mail: jarmiller@control.ee.ethz.ch).}
\thanks{$^2$ J. Miller, J. Zheng, and M. Sznaier are with the Robust Systems Lab,  ECE Department, Northeastern University, Boston, MA 02115. (e-mails: zheng.jian1@northeastern.edu, msznaier@coe.neu.edu)}
\thanks{$^3$
Naval Undersea Warfare Center 
1176 Howell St, Newport, RI 02841 (e-mail: chixenbaugh@umassd.edu). }
\thanks{J. Miller was partially supported by Swiss National Science Foundation Grant 200021\_178890. J. Miller, J. Zheng, and M. Sznaier were partially supported by NSF grants  CNS--1646121, ECCS--1808381 and CNS--2038493, AFOSR grant FA9550-19-1-0005, and ONR grant N00014-21-1-2431.  
}}

\begin{document}

\maketitle

\begin{abstract}
\label{sec:abstract}
This paper focuses on the stabilization and regulation of linear systems affected by quantization in state-transition data and actuated input. The observed data are composed of tuples of current state, input, and the next state's interval ranges based on sensor quantization. Using an established characterization of input-logarithmically-quantized stabilization based on robustness to sector-bounded uncertainty, we formulate a nonconservative infinite-dimensional linear program that enforces superstabilization of all possible consistent systems under assumed priors.
We solve this problem by posing a pair of exponentially-scaling linear programs, and demonstrate the success of our method on example quantized systems.



\end{abstract}
\section{Introduction}
\label{sec:introduction}

This paper performs \ac{DDC} of discrete-time linear systems under data quantization in the state-transition records and logarithmic quantization in the input. Input quantization can be encountered in data-rate constraints for network models when sending instructions to digital actuators, and its presence adds a nonlinearity to system dynamics \cite{brockett2000quantized, elia2001stabilization, nesic2009unified}.

The logarithmic input quantizer offers the coarsest possible quantization density \cite{elia2001stabilization} among all possible quantization schemes. These logarithmic quantizers admit a nonconservative characterization as a Lur\'{e}-type sector-bounded input \cite{fu2005sector, fu2009finite, zhou2013adaptive}. 
Data quantization could occur in the storage of sensor data into bits on a computer, and admits the mixed-precision setting of sensor fusion with different per-sensor precisions.




\ac{DDC} is a design method to synthesize control laws directly from acquired system observations and model/noise priors, without first performing  system-identification/robust-synthesis pipeline \cite{HOU20133, formentin2014comparison, hou2017datasurvey}. This paper utilizes a Set-Membership approach to \ac{DDC}: furnishing a controller along with a certificate that the set of all quantized data-consistent plants are contained within the set of all commonly-stabilized plants. Certificate methods for set-membership \ac{DDC} approaches include Farkas certificates for polytope-in-polytope containment \cite{cheng2015robust, miller2023ddcpos}, a Matrix S-Lemma for \acp{QMI} to prove quadratic and robust stabilization \cite{waarde2020noisy, van2023quadratic, miller2022lpvqmi}, and Sum-of-Squares certificates of polynomial nonnegativity \cite{dai2020semi, martin2021data, miller2022eiv_short, zheng2023robust}.

Other methods for \ac{DDC} include Iterative Feedback Tuning \cite{hjalmarsson1998iterative}, Virtual Reference Feedback Tuning \cite{campi2002virtual, formentin2012non}, Behavioral characterizations (Willem's Fundamental Lemma) with applications to Model-Predictive Control \cite{willems2005note, depersis2020formulas, coulson2019data, berberich2021robustmpc}, moment proofs for switching control \cite{dai2018moments}, learning with Lipschitz bounds \cite{robey2020learning, lindemann2021learning}, and kernel regression \cite{thorpe2023physicsinformed}.

The most relevant prior work to the quantized \ac{DDC} approach in this paper is the research in \cite{zhao2022data}. The work in \cite{zhao2022data} performs utilizes the approach of \cite{fu2005sector} by treating logarithmic-quantizing control as an $H_\infty$ small-gain task. They then formulate the consistency set of data-plants as \iac{QMI}, and use the Matrix S-Lemma \cite{waarde2020noisy} to certify common stabilization. In contrast, our work includes quantized data as well as quantized control by developing a polytopic description of the plant consistency set. We then restrict to superstabilization \cite{polyak2001optimal, polyak2002superstable} to formulate \ac{DDC} \acp{LP} over the polytopic consistency set. In the case of quantization of data, the \ac{QMI} approach in \cite{zhao2022data} would then over-approximate the polytopic consistency constraint with a single ellipsoidal region.


The contributions of this work are:

\begin{itemize}
    \item A formulation for superstabilizing \ac{DDC} under input and data quantization
    \item A sign-based \ac{LP} for  data-driven quantized superstabilization that grows exponentially in $n$ and $m$
    \item A more tractable \ac{AARC} that is exponential in $m$ alone.
    
\end{itemize}

This paper has the following structure: 
Section \ref{sec:preliminaries} introduces notation and superstabilization. Section \ref{sec:quantization}
provides an overview of the data and logarithmic-input quantization schemes considered in this work. Section \ref{sec:quantized_ddc} formulates superstabilizing \ac{DDC} under quantization as a pair of equivalent \acp{LP}. 
Section \ref{sec:examples} demonstrates these algorithms on example quantized systems. Section \ref{sec:conclusion} concludes the paper.


\section{Preliminaries}
\label{sec:preliminaries}

\begin{acronym}[WSOS]

\acro{AARC}{Affinely-Adjustable Robust Counterpart}
\acroindefinite{AARC}{an}{a}

\acro{DDC}{Data Driven Control}





\acro{LP}{Linear Program}
\acroindefinite{LP}{an}{a}







\acro{QMI}{Quadratic Matrix Inequality}
\acroplural{QMI}[QMIs]{Quadratic Matrix Inequalities}



\end{acronym}

\subsection{Notation}

\begin{tabular}{p{0.15\columnwidth}p{0.75 \columnwidth}}
$a..b$ & Natural numbers between $a$ and $b$\\
$\R^n$ & $n$-dimensional real Euclidean space\\
$\R^n_{\geq 0} \ (\R^n_{> 0})$ & $n$-dimensional nonnegative  (positive) orthant\\
$\R^{n\times m}$ & $n\times m$-dimensional real matrix space\\
$\1_n, \ \0_n$ & Vector of all ones or zeros\\
$I_n$ & Identity matrix \\
$\otimes$ & Kronecker product\\
$\text{vec}(X)$ & Column-wise vectorization of a matrix\\
$X^T$ & Matrix transpose \\
$\norm{x}_\infty$ & $L_\infty$-norm (vector): $\max_i \abs{x}_i $\\
$\norm{X}_\infty$ & Induced $L_\infty$ norm (matrix): $\max_i \sum_j \abs{X_{ij}}$ \\
$x./y$ & Element-wise division between $x$ and $y$ \\
$A \leq B$ & Element-wise $\leq$ between $A, B \in \R^{n \times m}$
\end{tabular}

\subsection{Superstabilization}
\label{sec:superstability}
A discrete-time system $x_+ = A x$ is \textit{Extended Superstable} if there exists nonnegative weights $v > 0$  such that $\norm{x./v}_\infty$ is a Lyapunov function \cite{polyak2004extended}. This condition may be expressed using an operator norm through the definition $Y = \text{diag}(v)$ and the constraint  $\norm{Y A Y^{-1}}_\infty < 1$. Standard \textit{superstability} is the restriction of extended superstability when $v=\1_n$.

A discrete-time linear system with input of
\begin{align}
    x^{+} &= A x + B u \label{eq:dyn_discrete}
\end{align}
is extended-superstabilized by the full-state-feedback controller $u = K x$ if there exists \cite{polyak2004extended}
$v \in \R_{>0}^n, \ S \in \R^{m \times n}$ with
\begin{align}
& \forall i \in 1..n, \ \alpha \in \{-1, 1\}^n: \nonumber \\
    &\qquad \textstyle  \sum_{j=1}^n \alpha_i \left(A_{ij}v_j + \sum_{k=1} B_{ik}S_{kj}\right) < v_i. & &\label{eq:ess_sign}
\end{align}

The controller $K$ forming the input $u=Kx$ is then recovered by $K = S \diag{1./v}$.
Problem \eqref{eq:ess_sign} is a set of $n 2^n$ strict linear inequality constraints. A more efficient method of imposing extended-superstability is by introducing a new matrix $M \in \R^{n\times n}$ \cite{yannakakis1991expressing}, 
\begin{subequations}
\label{eq:ess_standard}
\begin{align}
    &\textstyle \sum_{j=1}^n M_{ij} < v_i & & \forall i \in 1..n  \\
    &\textstyle  \abs{A_{ij}v_j + \sum_{k=1} B_{ik}S_{kj}} \leq M_{ij} & &\forall i,j \in 1..n.
\end{align}
\end{subequations}
Problems \eqref{eq:ess_sign} and \eqref{eq:ess_standard} are equivalent, in which an admissible selection for $M$ is $M_{ij} = \abs{A_{ij}v_j + \sum_{k=1} B_{ik}S_{kj}}.$
The conditions in \eqref{eq:ess_sign} and \eqref{eq:ess_standard} is necessary and sufficient for full-state feedback extended superstabilization.

If the system in \eqref{eq:dyn_discrete} is superstabilized $(v=\1_n)$ and $\norm{A + BK}_\infty \leq \lambda$ with $\lambda < 1$, then any closed-loop trajectory $x_t$ starting at $x_0$ with $\forall t: u_t =0$ will satisfy $\norm{x_t}_\infty \leq \lambda^t \norm{x_0}_\infty$ \cite{SZNAIER19963550, polyak2001optimal}. The quantity $\lambda$ can be interpreted as a decay rate, and the controller $K$ can be designed using \iac{LP} to minimize $\lambda$ and ensure the fastest possible convergence. A similar minimal peak-to-peak design task for extended superstabilization requires the solution of parametric \ac{LP} with a single free parameter \cite{polyak2004extended}.




\section{Quantization}
\label{sec:quantization}

This section will introduce the two sources of quantization considered in this paper. 

\subsection{Quantization of Data}

Our data $\dc$ with $N_s$ samples is composed of the current state $\hat{x}$, input $\hat{u}$, and bounds on the subsequent state $[p, q]$, forming the $N_s$ tuples $\dc = \cup_{s=1}^{N_s} (\hat{x}_s, \hat{u}_s, p_s, q_s)$. 
We define the polytope $\pc(A, B)$ as the set of all plants that are consistent with the data in $\dc$:
\begin{align}
    \pc = \{(A, B) \mid \forall s \in 1..N_s: A \hat{x}_s + B \hat{u}_s  & \in [p_s, q_s] \}. \label{eq:residual_constraints}
\end{align}


The bounds $p_s, q_s$ at each sample-index $s$ may arise from interval quantization. In the case where a quantization process performs rounding to the first decimal place, the true state transition $x^+ = 0.368$ would be restricted to the interval to the interval described by $p = 0.3$ and $q = 0.4$. 



This data-quantization framework in $\dc$ allows for the integration of $L_\infty$-bounded process-noise.
    In the case where there exists a process-noise $w_s$ such that $A \hat{x}_s + B \hat{u}_s + w_s \in [p_s, q_s]$ with $\norm{w_s}_\infty \leq \epsilon$, interval arithmetic can be used to express the data constraint as $A \hat{x}_s + B \hat{u}_s \in [p_s-\epsilon, q_s+\epsilon]$. 
    

\subsection{Quantization of Input}

A scalar logarithmic quantizer with density $\rho \in (0, 1)$ and step $\delta = (1-\rho)/(1+\rho)$ is defined by $g_\rho: \R \rightarrow \R$ \cite[Equation 7]{fu2005sector}:
\begin{align}
    g_\rho(z) = \begin{cases} \rho^i & \exists i \in \N\mid \frac{1}{1+\delta} \rho^i \leq z \leq \frac{1}{1-\delta} \rho^i \\
    0 & z=0 \\
    -g_\rho(-z) & v < 0.\end{cases} \label{eq:log_quantize}
\end{align}
We will obey the convention of \cite{fu2005sector} in referring to $\rho$ as the \textit{quantization density}, in which a larger $\rho$ refers to a coarser quantizer with wider intervals.
A $\rho$-logarithmically-quantized linear system has dynamics
\begin{align}
    x_{t+1} &= A x_t + B g_\rho(u_t), \label{eq:dyn_quantize}
\end{align}
where the quantization in $g_\rho$ should be understood to occur elementwise in $u_t$. 

The following proposition establishes a sector-bound characterization of logarithmic quantization (for $m=1$):
\begin{prop}[Eq. (21)-(22) in \cite{fu2005sector}]
\label{prop:quant_error}
    For any $z \geq 0$ and logarithmic quantization density $\rho>0$ with 
    \begin{equation}
        \delta = (1-\rho)/(1+\rho), \label{eq:delta_rho}
    \end{equation} the quantization error at $z$ satisfies a multiplicative bound
    \begin{equation}
        z - g_\rho(z) \in [-\delta z, \delta z]. \label{eq:mult_error_bound}
    \end{equation}
\end{prop}

Figure \ref{fig:log_quantizer} plots the graph of a logarithmic quantizer with $\rho = 0.4, \ \delta = 0.4286$ along with the error bound in \eqref{eq:mult_error_bound} over the interval $u \in [-8, 8]$.
\begin{figure}[h]
    \centering
    \includegraphics[width=\exfiglength]{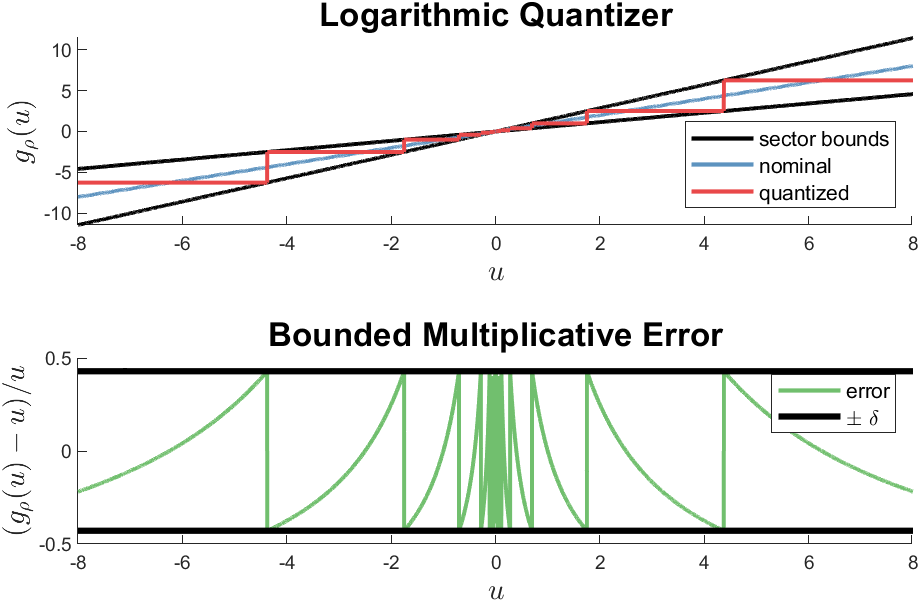}
    \caption{Logarithmic quantizer $(\rho=0.4)$ and error bound}
    \label{fig:log_quantizer}
\end{figure}

The trajectories of a logarithmically quantized systems with $m=1$ are therefore contained in the class of scalar-$\Delta$ sector-bounded models:
\begin{align}
    x_{t+1} &= \left[A + (1+\Delta) B K \right]x_t & 
 & \forall \Delta \in [-\delta, \delta].\label{eq:dyn_quantize_feedback}
\end{align}

Theorem 2.1 of \cite{fu2005sector} proves that the state-feedback controller $u_t = K x_t$  with $K \in \R^{1 \times n}$ \textbf{quadratically} stabilizes \eqref{eq:dyn_discrete} iff $u=Kx$ can \textbf{quadratically} stabilize \eqref{eq:dyn_quantize}.

For systems in which each input channel $u_j$ has a separate quantization density $(\rho_j, \delta_j)$, quadratic state-feedback stabilization of the quantized system will occur if 
\cite[Theorem 3.2]{fu2005sector}:
\begin{align}
& \forall \Delta \in \textstyle \prod_{j=1}^m [-\delta_j, \delta_j] \nonumber \\ 
    & \qquad x_{t+1} = \left[A +  B(I_m+\diag{\Delta}) K \right]x_t. \label{eq:dyn_quantize_feedback_vector}
\end{align}

The work in \cite{fu2005sector} and \cite{zhao2022data} treat common stabilization of \eqref{eq:dyn_quantize_feedback} as an $H_\infty$ optimization using the small-gain theorem for a  sector-bounded uncertainty. The muli-input small-gain formulation in \eqref{eq:dyn_quantize_feedback_vector} is posed solved using a conservative multi-block S-Lemma.


\subsection{Combined Superstability and Input-Quantization}
We can apply extended superstabilization method Section \ref{sec:superstability} towards the control of input-quantized systems, as represented by the sector-bounded model class in \eqref{eq:dyn_quantize_feedback_vector}.
\begin{thm}
\label{thm:quant_ess}
    A logarithmically quantized  system in \eqref{eq:dyn_quantize} is extended superstablized by a controller $u= K x$ if there exists a $v \in \R_{>0}^n, \ S \in \R^{m \times n}, M \in \R^{n \times n}$ with $Y = \diag{v}$ 
\begin{subequations}
\label{eq:ess_quantized}
\begin{align}
& \forall i \in 1..n \nonumber \\
    & \qquad \textstyle \sum_{j=1}^n M_{ij} < v_i & &    \\
    & \forall \Delta \in \textstyle \prod_{j=1}^m [-\delta_j, \delta_j] \nonumber  \\
    &\textstyle -M \leq A Y+  B (I_m + \diag{\Delta}) S \leq M. \label{eq:ess_quantized_con}    
    \end{align}
\end{subequations}
The recovered controller is $K = S Y^{-1}.$
\end{thm}
\begin{proof}
In the case where $\Delta=0$, then the quantized program \eqref{eq:ess_quantized} is equivalent to the unquantized program \eqref{eq:ess_standard}. We can apply Proposition \ref{prop:quant_error} to generate a sector-bound description of quantization, together with separate input channel quantization based on Equation \eqref{eq:dyn_quantize_feedback_vector} regarding the multiplicative perturbations $\Delta$. The linear inequality constraints \eqref{eq:ess_quantized} are convex,  such that a common $M$ is a worst-case certificate over the all possible closed-loop matrices $AY + B(I_m + \diag{\Delta}) S \leq M$. Such a certificate ensures extended superstability of all systems in \eqref{eq:dyn_quantize_feedback}.
\end{proof}

\begin{cor}
    We can enumerate the convex constraint \eqref{eq:ess_quantized} over the vertices of the hypercube formed by $\delta$, producing the equivalent statement of 
    \begin{align}& \forall \gamma \in \textstyle\prod_{j=1}^m\{-\delta_j, \delta_j\} \nonumber  \\
            &\textstyle -M \leq A Y+  B (I_m +  \diag{\gamma}) S \leq M.\label{eq:ess_quantized_enum} 
    \end{align}
\end{cor}

\begin{cor}
    An equivalent formulation to \eqref{eq:ess_quantized} with respect to sign-enumeration in \eqref{eq:ess_sign} and substitution $\beta =  1+\gamma$ is the following \ac{LP} in $n 2^{n+m}$ constraints:
    \begin{align}
& \forall i \in 1..n, \ \alpha \in \{-1, 1\}^n, \ \beta \in \textstyle\prod_{j=1}^m\{1-\delta_j, 1+\delta_j\} : \nonumber \\
    &\quad \textstyle  \sum_{j=1}^n \alpha_j \left( A_{ij}v_j + \sum_{k=1}^m \beta_k  B_{ik} S_{kj} \right)< v_i. & &\label{eq:ess_sign_quantized}
\end{align}
\end{cor}

\begin{prop}
    A controller $K$ that is feasible for quantization $\delta 
 \in \R_{\geq 0}^m$ in \eqref{eq:ess_sign_quantized} will also be feasible for every $\delta' \in \R_{\geq 0}^m$ with $\delta' \leq \delta$.
\end{prop}
\section{Quantized DDC}
\label{sec:quantized_ddc}

This section will outline \iac{DDC} approach towards quantized superstability. 
    
Given data in $\dc$, let $\mathcal{P}$ in \eqref{eq:residual_constraints}
be the polytopic consistency of plants $(A, B)$ in agreement with $\dc$.

Our task is to solve the following problem:
\begin{prob}
    Find a state-feedback controller $u = K x$ such that the quantized system \eqref{eq:dyn_quantize_feedback} is (extended) superstable for all $(A, B) \in \pc$.
    \label{prob:ddc_quantized}
\end{prob}

\subsection{Consistency Polytope Representation}

Let us define $\bx, \bu, \mathbf{p}, \mathbf{q}$ as the following concatenations of data in $\dc$:
\begin{subequations}
\label{eq:data_matrices}
\begin{align}
    \bx &= \begin{bmatrix}
        \hat{x}_1; & \hat{x}_2; & \ldots & \hat{x}_{N_s}
    \end{bmatrix} \\
    \bu &= \begin{bmatrix}
        \hat{u}_1; & \hat{u}_2; & \ldots & \hat{u}_{N_s}
    \end{bmatrix}\\ 
    \mathbf{p} &= \begin{bmatrix}
        p_1; & p_2; & \ldots & p_{N_s} 
    \end{bmatrix}\\
    \mathbf{q} &= \begin{bmatrix}
        q_1; & q_2; & \ldots & q_{N_s}
    \end{bmatrix}.
\end{align}
\end{subequations}

The data-consistency polytope in \eqref{eq:residual_constraints} may be represented using the data matrices in \eqref{eq:data_matrices} as
\begin{align}    
    G_\dc &= \begin{bmatrix}
        -\bx^T \otimes  I_n & -\bu^T \otimes I_n \\
        \bx^T \otimes  I_n & \bu^T \otimes I_n
    \end{bmatrix} \nonumber\\
    h_\dc &= \begin{bmatrix} - \mathbf{p}; &  \mathbf{q} \end{bmatrix}\label{eq:polytope_data}\\    
    \pc &= \{(A, B) \mid G_\dc [\textrm{vec}(A); \textrm{vec}(B)] \leq h_\dc \}, \nonumber
\end{align}
using the Kronecker identity $\textrm{vec}(P X Q) = (Q^T \otimes P)\textrm{vec}(X)$ for matrices $(P, X, Q)$ of compatible dimensions. 
We will denote $L \leq 2 n N_s$ as the number of faces in \eqref{eq:polytope_data} ($h_\dc \in \R^{1 \times L}$).
The number of faces $L$ can be reduced from $2n N_s$ by pruning redundant constraints from $\pc$ \cite{caron1989degenerate} through iterative \acp{LP}.

\subsection{Sign-Based Approach}

The sign-based program in \eqref{eq:ess_sign_quantized} in the \ac{DDC} case can be considered as a finite-dimensional robust \ac{LP}:
    \begin{align}
& \forall i \in 1..n, \ \alpha \in \{-1, 1\}^n, \ \beta \in \textstyle\prod_{j=1}^m\{1-\delta_j, 1+\delta_j\}:  \label{eq:ess_sign_quantized_ddc} \\
    &\quad \textstyle  \sum_{j=1}^n \alpha_j \left(A_{ij}v_j + \sum_{k=1} \beta_k  B_{ik} S_{kj}\right) < v_i, \ \forall (A, B) \in \pc. \nonumber
\end{align}

Program \eqref{eq:ess_sign_quantized_ddc} features a total of $n 2^{n+m}$ strict robust inequalities. We will add a stability tolerance $\eta > 0$ in order to modify the comparator and right-hand side of \eqref{eq:ess_sign_quantized_ddc} into a nonstrict inequality $\leq v_i - \eta$.
Each nonstrict robust inequality in $\alpha, \beta$ may be formulated as a polytope:

\begin{align}    
    G_{\alpha \beta} &= \begin{bmatrix}
        (\diag{v} \alpha)^T \otimes I_n & (\diag{\beta} S \alpha)^T \otimes I_n
    \end{bmatrix} \nonumber \\
    h_{\alpha \beta} &= v - \eta \1 \label{eq:polytope_sign_stab}\\    
    \pc_{\alpha\beta} &= \{(A, B) \mid G_{\alpha \beta} [\textrm{vec}(A); \textrm{vec}(B)] \leq
    h_{\alpha \beta} \}. \nonumber
\end{align}

We will enforce containment of $\pc$ in each $\pc_{\alpha \beta}$ using the Extended Farkas Lemma:

\begin{lem}[Extended Farkas Lemma \cite{hennet1989farkas, henrion1999control}]
\label{lem:ext_farkas}
Let $P_1 = \{x \mid G_1 x \leq h_1\}$ and $P_2 = \{x \mid G_2 x \leq h_2\}$ be a pair of polytopes with $G_1 \in \R^{m \times n}$ and $G_2 \in \R^{p \times n}$. Then $P_1 \subseteq P_2$ if and only if there exists a matrix $Z \in \R^{p \times m}_{\geq 0}$ such that,
\begin{align}
    Z G_1 &= G_2, & Z h_1 \leq h_2.
\end{align}
\end{lem}
\begin{rmk}
    The Extended Farkas Lemma is a particular instance of a robust counterpart \cite{ben2009robust} when certifying validity of a system of linear inequalities over polytopic uncertainty.
\end{rmk}

A sign-based program to solve Problem \ref{prob:ddc_quantized} is:
\begin{subequations}
\label{eq:lp_stab_sign}
\begin{align}
    \find_{v, S, Z} \quad & \forall \alpha \in \{-1, 1\}^n, \ \beta \in \textstyle\prod_{j=1}^m\{1-\delta_j, 1+\delta_j\}: \\
    & \qquad Z_{\alpha \beta} G_\dc = G_{\alpha \beta}, \quad Z_{\alpha \beta} h_\dc \leq  h_{\alpha \beta} \label{eq:lp_stab_ext}\\
    & \qquad Z_{\alpha \beta} \in \R_{\geq 0} ^{n \times L} \\
    & v-\eta \1_n \in \R_{\geq 0}^n, \ S \in \R^{n \times m}. \label{eq:lp_stab_var}
\end{align}
\end{subequations}

\subsection{Lifted Approach}

We can solve Problem \ref{prob:ddc_quantized} by posing \eqref{eq:ess_quantized} as an infinite-dimensional \ac{LP} in terms of a function $M: \pc \rightarrow \R^{n \times n}$.

\begin{thm}
    A state-feedback controller $u = Kx$ will solve Problem \ref{prob:ddc_quantized} if the following infinite-dimensional \ac{LP} has a feasible solution with $v \in \R_{>0}^n, \ S \in \R^{m \times n}, M: \pc \rightarrow \R^{n \times n}$ with $Y = \diag{v}$
    \begin{subequations}
\label{eq:ess_quantized_ddc}
\begin{align}
& \forall i \in 1..n \nonumber \\
    & \qquad \textstyle \sum_{j=1}^n M_{ij}(A, B) < v_i & & \label{eq:ess_quantized_constraint_1}   \\
    & \forall \beta \in  \textstyle\prod_{j=1}^m \{1-\delta_j, 1+\delta_j\} \nonumber  \\
    &\textstyle -M(A, B) \leq A Y+  B ( \diag{\beta}) S \leq M(A, B).  \label{eq:ess_quantized_constraint}    
    \end{align}
\end{subequations}
\end{thm}
\begin{proof}
    Each plant $(A, B) \in \pc$ has a certificate of extended superstabilizability $(v, M(A, B))$ by Theorem \ref{thm:quant_ess}. 
    If \eqref{eq:ess_quantized_ddc} is feasible, then 
    all plants in $\pc$ simultaneously extended superstabilized by a common $K = S Y^{-1}$ .
\end{proof}

\begin{rmk}
    The function $M(A, B)$ may be treated as an adjustable decision variable given the a-priori unknown $(A, B) \in \pc$ \cite{yanikouglu2019survey}.
\end{rmk}




The infinite-dimensional \ac{LP} in \eqref{eq:ess_quantized_ddc} must be truncated into a finite-dimensional convex program in order to admit computationally tractable formulations. One method to perform this truncation is to restrict $M(A, B)$ to an affine function by defining $M^0, M^{A}_{ij}, M^{B}_{ik} \in \R^{n \times n}$ to form
\begin{align}
    M(A, B) &= \textstyle M^0 + \sum_{i j} M^{A}_{ij}  A_{ij} + \sum_{i k} M^{B}_{ik}  B_{ik}, \label{eq:m_affine}
    \end{align}

We can define the quantities $\mathbf{m} = (m^0, m^a, m^b)$
\begin{subequations}
\begin{align}
    m^0 &= \textrm{vec}(M^0)\\ 
    m^a &= \begin{bmatrix}
        \textrm{vec}(M^A_{11}), & \textrm{vec}(M^A_{21}), & \ldots, & \textrm{vec}(M^A_{nn})
    \end{bmatrix} \\
    m^b &= \begin{bmatrix}
        \textrm{vec}(M^B_{11}), & \textrm{vec}(M^B_{21}), & \ldots, & \textrm{vec}(M^B_{nm})
    \end{bmatrix}.
\end{align}
\end{subequations}
in order to obtain a vectorized expression for \eqref{eq:m_affine} with 
    \begin{align}
        \textrm{vec}(M(A, B)) &= m^0 + m^a \textrm{vec}(A) + m^b \textrm{vec}(B). \\
        \intertext{The row-sums of $M$ can be expressed as }
                \textrm{vec}(M(A, B) \1_n ) &= (\1_n^T \otimes I_n) (m^0 + m^a \textrm{vec}(A))
                \\
                &\qquad +  (\1_n^T \otimes I_n) (m^b \textrm{vec}(B)).\nonumber
    \end{align}

    The constraint in \eqref{eq:ess_quantized_constraint_1} with stability factor $\eta>0$ can be reformulated as membership in the following polytope $\pc_M$:
    \begin{align}
            G_M &= (\1_n^T \otimes I_n)\begin{bmatrix}
        m^a, & m^b
    \end{bmatrix} \nonumber \\
    h_M &= [v - \eta - (\1_n^T \otimes I_n)m^0]  \\  
    \pc_{M} &= \{(A, B) \mid G_{M} [\textrm{vec}(A); \textrm{vec}(B)] \leq
    h_{M} \}. \nonumber
    \end{align}

    

The polytopic constraint region in \eqref{eq:ess_quantized_constraint} for each $\beta \in \textstyle\prod_{j=1}^m\{1-\delta_j, 1+\delta_j\}$ is
\begin{align}    
    G_{\beta}^s &= \begin{bmatrix}
        -m^a - \diag{v}^T\otimes I_n & -m^b- (\diag{\beta} S)^T \otimes I_n \\
        -m^a + \diag{v}^T\otimes I_n & -m^b+ (\diag{\beta} S)^T \otimes I_n
    \end{bmatrix} \nonumber \\    
    h_{\beta} &= \begin{bmatrix}
         m^0; & m^0
    \end{bmatrix} \label{eq:polytope_sign_stab_aarc}\\    
    \pc_{\beta} &= \{(A, B) \mid G_{\beta} [\textrm{vec}(A); \textrm{vec}(B)] \leq
    h_{\beta} \}. \nonumber
\end{align}

The affine restriction of $M$ in \eqref{eq:m_affine} results in \iac{AARC} program for \eqref{eq:ess_quantized_ddc}:
\begin{subequations}
\label{eq:lp_stab_aarc}
\begin{align}
    \find_{v, S, Z, \mathbf{m}} \quad & \forall \beta \in \textstyle\prod_{j=1}^m\{1-\delta_j, 1+\delta_j\}: \\
    & Z_{\beta} G_\dc = G_{\beta}, \quad Z_{\beta} h_\dc \leq  h_{\beta} \label{eq:lp_stab_ext_aarc}\\
    & Z_{\beta} \in \R_{\geq 0} ^{2 n^2 \times L} \\
    & Z_M G_\dc = G_M, \quad Z_M h_\dc \leq h_M \\
    & Z_M \in \R^{n \times L}_{\geq 0} \\
    &m^0 \in \R^{n^2 \times 1}, \   m^A \in \R^{n^2 \times n^2}\\
    &m^B \in \R^{n^2 \times nm} \\
    & v-\eta \1_n \in \R_{\geq 0}^n, \ S \in \R^{n \times m}. \label{eq:lp_stab_var_aarc}
\end{align}
\end{subequations}

\subsection{Computational Complexity}

We will quantify the computational complexity  \eqref{eq:lp_stab_sign} and \eqref{eq:lp_stab_aarc} based on the number of robust inequalities (for \eqref{eq:ess_sign_quantized_ddc} and \eqref{eq:ess_quantized_ddc}), scalar variables $(v, S, Z, \mathbf{m})$, slack variables/constraints introduced in reformulations of scalar inequality constraints (e.g., $v - \eta \1_n \in \R^n_{\geq 0} \mapsto q \in \R^n_{\geq 0}, \ v-\eta \1_n = q$), scalar inequality constraints $(\in \R_{\geq 0})$, and scalar equality constraints. These counts (up to the highest order terms to save space) are listed in Table \ref{tab:complexity}.


\begin{table}[h]
    \centering
        \caption{Comparison  between \acp{LP} \eqref{eq:lp_stab_sign} and \eqref{eq:lp_stab_aarc}}
    \label{tab:complexity}
    \begin{tabular}{r|l l}
         & sign-based \eqref{eq:lp_stab_sign} & \ac{AARC} \eqref{eq:lp_stab_aarc} \\
        robust ineq. &  $n 2^{n+m}$& $n + n^2 2^{m+1}$\\
        scalar vars. & $n(m+1+ 2^{n+m}L)$ & $n(m+1) + n^2(1+nm+n^2)$ \\
        & & $\quad + n^2 L 2^{m+1}$\\
        slack vars. & $n+2^{n+m+1} n^2$ & $n(1+2^{m+1})$\\ 
        eq. cons. & $n^2(n+m) 2^{n+m} + n$& $(2^{m+1} n^2+n)(n+m) + n$ \\
        ineq. cons. & $2^{n+m}(n + nL) + n$& $n(L+2) + n^2 (L+1)2^{m+1}$ 
    \end{tabular}
\end{table}


Note how $n$ appears exponentially in the  sign-based scheme \eqref{eq:lp_stab_sign}, while $n$ enters only polynomially for quantities in the \ac{AARC} \eqref{eq:lp_stab_aarc}.

Given that the running-time of an Interior Point Method for $N$-variable \acp{LP} up $\gamma$-optimality is approximately $O(N^{\omega+0.5} \abs{\log(1/\gamma)})$ (with matrix multiplication constant $\omega$) \cite{wright1997primal}, the \ac{AARC} is more computationally efficient than the sign-based scheme as $n$ increases.


\section{Numerical Examples}

\label{sec:examples}

MATLAB (2021a) code to execute all examples is publicly available
\footnote{\url{https://github.com/Jarmill/quantized_ddc}}. The convex optimization problems  \eqref{eq:lp_stab_sign} and \eqref{eq:lp_stab_aarc} are modeled in YALMIP \cite{lofberg2004yalmip} 
(including the robust programming module \cite{lofberg2012} with option `\texttt{lplp.duality}') 
and solved in Mosek 9.2 \cite{mosek92}.


\subsection{3-state 2-input}

The first example will involve superstabilization of the following system 3-state 2-input discrete-time linear system:
\begin{subequations}
\label{eq:sys1}
\begin{align}
    A &= \begin{bmatrix}
        -0.1300&-0.3974&0.2030\\
-0.3974&-0.5000&0.2990\\
0.2030&0.2990&-0.5262
    \end{bmatrix}, \\
    B &= \begin{bmatrix}
        0.2179&1.2300\\
0.3592&0\\
-1.1553&0
    \end{bmatrix}.
\end{align}
\end{subequations}

System \eqref{eq:sys1} is open-loop unstable with eigenvalues of $[-1.0185, -0.2613, 0.1236].$

We collect $T=100$ input-state-transition observations of system \eqref{eq:sys1} to form $\dc$. The transition observations are quantized according to the following partition with 9 bins:
\begin{align}
    (-\infty, -4] \cup [-4, -3] \cup [-3, -2] \ldots [3, 4] \cup [4, \infty). \label{eq:partition1}
\end{align}

Superstabilization $(v=\1_3)$ is performed by solving the sign-based scheme in \eqref{eq:ess_sign_quantized_ddc}. An objective is added to minimize $\lambda \in \R$ such that $\forall i: \sum_{j} M_{ij} \leq \lambda$, in which $\lambda < 1$ indicates a successful worst-case superstabilization under input and data quantization. 

Figure \ref{fig:ss_rho_vs_lam} plots worst-case optimal values of $\lambda$ as a function of the quantization density $\rho$, in which $\rho$ is the same for all inputs. The $T=60$ data preserves the first 60  elements of the 100 observations in $\dc$ (with a similar process for $T=80$). Gain values for the ground truth (model-based case when \eqref{eq:sys1} is known) are presented as a comparison. We note that $\rho \rightarrow 1$ results in $\delta \rightarrow 0$ by \eqref{eq:delta_rho}, for which the (limiting) quantization law at $\rho=1$ is $g_{\rho=1}(u) = u$.

\begin{figure}[h]
    \centering
    \includegraphics[width=0.8\exfiglength]{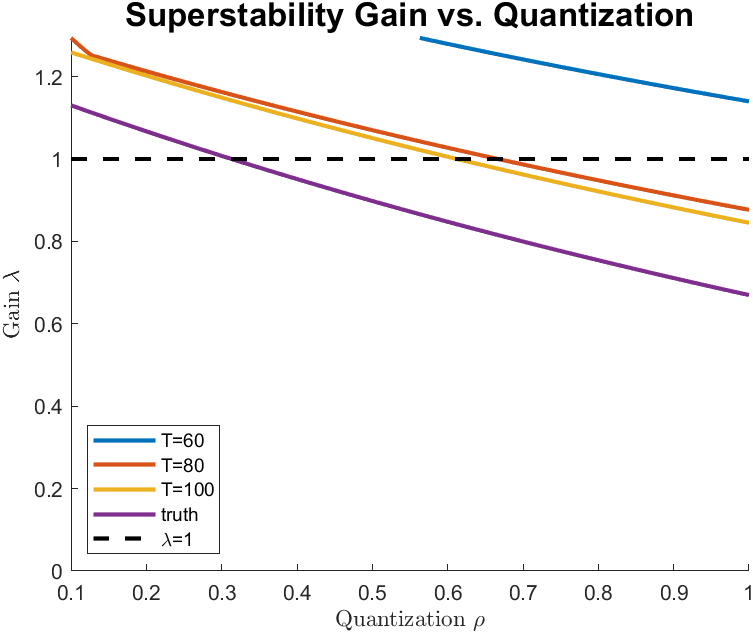}
    \caption{Peak-to-peak gain $(\lambda)$ vs.  quantization density $(\rho)$}
    \label{fig:ss_rho_vs_lam}
\end{figure}


Table \ref{tab:min_rho} lists the minimal feasible $\rho$ (up to four decimal places) such that the sign-based formulation in \eqref{eq:ess_sign_quantized_ddc} returns a feasible superstabilizing (SS) or extended superstabilizing (ESS) controller. The symbol $\varnothing$ indicates primal infeasiblility of the \ac{LP} for all $\rho \leq 1$.

\begin{table}[h]
    \centering
        \caption{Minimal $\rho$ with sign-based of \eqref{eq:sys1}}
    \label{tab:min_rho}
    \begin{tabular}{r|c c c c}
        $T$ & 60 & 80 & 100 & Truth\\ \hline
        SS &  $\varnothing$ & 0.6727 & 0.6182 & 0.3182 \\
        ESS & 0.9397 & 0.3494 & 0.2081 & 0.1422\\
        
    \end{tabular}
\end{table}

Table \ref{tab:min_rho_aarc} lists the minimal $\rho$ for \ac{AARC}-based quantized superstabilization. There is no difference between the ground-truth values in Tables \ref{tab:min_rho}  and \ref{tab:min_rho_aarc}, because the underlying finite-dimensional \acp{LP} with nonrobust inequality constraints are equivalent.
\begin{table}[h]
    \centering
        \caption{Minimal $\rho$ with \ac{AARC} stabilization of \eqref{eq:sys1}}
    \label{tab:min_rho_aarc}
    \begin{tabular}{r|c c c c}
        $T$ & 60 & 80 & 100 & Truth \\ \hline
        SS &  $\varnothing$ & $\varnothing$ & 0.9500 & 0.3182\\
        ESS & $\varnothing$ & $\varnothing$ & 0.7723 & 0.1422\\
        
    \end{tabular}
\end{table}


\subsection{5-state 3-input}

The second example performs extended superstabilization over the following system with 5 states and 3 inputs:
\begin{subequations}
\begin{align}
    A &= (1/5)[\min(i/j, j/i)]_{ij} + (1/2) I_{5},\\
    B &= [I_3; \0_{2 \times 3}].
\end{align}
\label{eq:sys2}
\end{subequations}

System \eqref{eq:sys2} is open-loop unstable with purely real eigenvalues of $[1.0633, 0.6507, 0.5502, 0.5046, 0.4812]$. The nominal system in \eqref{eq:sys2} can be extended-superstabilized until $\rho = 0.2245$.

The  $T=350$ state-input collected transitions of \eqref{eq:sys2}  are quantized according to the following partition with 26 bins:
\begin{align}
    (-\infty, -6] \cup [-6, -5.5] \cup [-5.5, -5] \ldots [5.5, 6] \cup [6, \infty). \label{eq:partition2}
\end{align}

The polytope in \eqref{eq:residual_constraints} has $2nT = 3500$ faces in $n(n+m)=40$ dimensions, of which $185$ of these faces are nonredundant.

We successfully solve the data-driven common-extended-superstabilizing \ac{AARC} program in \eqref{eq:lp_stab_aarc} at $\rho=0.8$  to acquire a feasible controller with parameters
    \begin{align}
        K &= -\begin{bmatrix}
            0.6434 & 0.0943 & 0.0785 & 0.0609 & 0.0330 \\
0.0965 & 0.6513 & 0.1409 & 0.0899 & 0.0842 \\
0.0650 & 0.1392 & 0.6528 & 0.1463 & 0.1183 \\
        \end{bmatrix} \nonumber \\
        v &= \begin{bmatrix}
            0.0137, &     0.0069  &  0.0058   & 0.0289   & 0.0289
        \end{bmatrix}.
    \end{align}
\section{Conclusion}

\label{sec:conclusion}

This paper presented a method to perform superstabilizing control of linear systems under state-transition data quantization and actuated-input quantization. The generated sign-based finite-dimensional \ac{LP} and lifted  infinite-dimensional \acp{LP} are nonconservative with respect to the common superstabilization task. This infinite-dimensional \ac{LP} has a number of constraints that is polynomial in the number of states $n$ and exponential in the number of inputs $m$. \Iac{AARC} was employed to truncate the infinite-dimensional \ac{LP}, in order to gain tractability at the expense of conservatism.

The logarithmic-quantization approach laid out in this paper involves an infinite number of quantization levels. Future work includes adapting the adaptive finite-level quantizing method of \cite{fu2009finite} for \ac{DDC}-superstabilization. Other investigations aim to decrease the computational impact of the presented scheme by formulating nonconservative \ac{LP} formulations that scale in a polynomial manner with $m$ rather than in an exponential manner, 
by reducing the conservatism of the \ac{AARC} truncation by allowing $M$ to be polynomial (using sum-of-squares certificates of nonnegativity), 
and by formulating control laws in the setting where $(\hat{x}, \hat{u})$ are also data-quantized (resulting in an Error-in-Variables model \cite{soderstrom2007errors} addressable by polynomial optimization \cite{miller2022eiv_short}).

\section*{Acknowledgements}

The authors would like to thank Roy Smith and the Automatic Control Lab of ETH Z\"{u}rich for their support.




\bibliographystyle{IEEEtran}
\bibliography{references.bib}

\end{document}